\documentclass[reqno, 12pt]{amsart}

\usepackage[utf8]{inputenc}
\usepackage[T1]{fontenc}
\usepackage{lmodern}
\usepackage{mathrsfs} 

\usepackage{amsmath}
\usepackage{amsthm}
\usepackage{amssymb}
\usepackage{mathscinet} 
\usepackage[all,cmtip]{xy} 
\usepackage{stmaryrd} 

\usepackage{enumerate} 
\usepackage{graphicx}
\usepackage{hyperref}
\usepackage[a4paper, hmargin=3cm, vmargin=3cm]{geometry}

\newcommand\mathens[1]{\mathbb{#1}} 

\newcommand{\ud}{\mathrm{d}}

\newcommand{\N}{\mathens{N}}
\newcommand{\Z}{\mathens{Z}}

\newcommand{\R}{\mathens{R}}
\newcommand{\RP}{\mathens{R}\mathrm{P}}

\newcommand{\CP}{\mathens{C}\mathrm{P}}
\newcommand{\HP}{\mathens{H}\mathrm{P}}
\newcommand{\CaP}{\mathens{C}\mathrm{aP}}

\newcommand\sphere[1]{\mathens{S}^{#1}}

\newcommand{\cont}{\textup{Cont}}
\newcommand{\contoc}{\textup{Cont}_{0}}

\newcommand{\id}{\textup{id}}

\DeclareMathOperator\im{im}

\DeclareMathOperator\crit{Crit}

\newtheorem{thm}{Theorem}[section]
\newtheorem{lem}[thm]{Lemma}
\newtheorem{cor}[thm]{Corollary}
\newtheorem{prop}[thm]{Proposition}
\newtheorem{question}{Question}
\newtheorem{prop-def}[thm]{Definition-proposition}

\theoremstyle{definition}
\newtheorem{definition}[thm]{Definition}

\theoremstyle{remark}

\begin{document}

\title{Morse estimates for translated points on unit tangent bundles}

\author[S. Allais]{Simon Allais}
\address{Simon Allais, 
Université Paris Cité and Sorbonne Université, CNRS, IMJ-PRG,
\newline\indent F-75013 Paris, France.}
\email{simon.allais@imj-prg.fr}
\urladdr{https://webusers.imj-prg.fr/~simon.allais/}
\date{May 27, 2022}

\begin{abstract}
    In this article, we study conjectures of Sandon on the minimal
    number of translated points in the special case of the unit
    tangent bundle of a Riemannian manifold.
    We restrict ourselves to contactomorphisms of $SM$ that lift
    diffeomorphisms of $M$ homotopic to identity.
    We prove that there exist sequences $(p_n,t_n)$ where $p_n$
    is a translated point of time-shift $t_n$ with $t_n\to+\infty$
    for a large class of manifolds.
    We also prove Morse estimates on the number of translated points
    in the case of Zoll Riemannian manifolds.
\end{abstract}
\maketitle

\section{Introduction}

In this article, we study conjectures essentially due to Sandon
on the minimal number of translated points in the special
case of the unit tangent bundle of a Riemannian manifold.
Let us recall the definition of translated points.
Let $(V^{2n-1},\alpha)$ be an oriented contact manifold with a fixed contact form $\alpha$
(\emph{i.e.} $\alpha\wedge(\ud\alpha)^{n-1}$ does not vanish).
A contactomorphism $\varphi\in\cont(V,\alpha)$ is a diffeomorphism of $V$
such that $\varphi^*\alpha = e^g\alpha$ for some $g:V\to\R$.
A point $p\in V$ is a discriminant point of $\varphi$ if and only if
it is fixed by $\varphi$ and $g(p)=0$ (this definition does not
depend on the choice of the contact form associated with $\ker\alpha$).
Let $(\phi^\alpha_t)$ be the Reeb flow of $\alpha$.
A point $p\in V$ is a translated point of $\varphi$ if and only if
it is a discriminant point of $\phi^\alpha_{-t}\circ\varphi$
for some $t\in\R$ called a time-shift of $p$.
Similarly to the Hamiltonian case, using a Weinstein neighborhood
of the graph of the identity, one can prove that for every
contactomorphism of a closed contact manifold
$\varphi\in\cont(V,\alpha)$ which is $C^1$-close
to the identity,
\begin{equation}\label{eq:Sandon}
    \# \{ p\in V\ |\ p \text{ is a translated point of } \varphi\}
    \geq \min_{f\in C^\infty(V,\R)} \# \crit(f),
\end{equation}
where $\crit(f)$ denotes the set of critical points of $f$ \cite{San13}.
Moreover, if the Reeb flow is periodic, this inequality is sharp
(see \emph{e.g.} the introduction of \cite{minimalLens}).
In \cite{San13}, Sandon proved that this inequality still holds without the
``$C^1$-close'' assumption, as long as $\varphi$ is isotopic to the identity
in the case of the real projective spaces $\RP^{2n-1}=\sphere{2n-1}/(z\sim -z)$ endowed with
the contact form induced by $\alpha_0 = \frac{1}{2}(x\ud y-y\ud x)\in\Omega^1(\sphere{2n-1})$.
Therefore, Sandon asked the following question,
where $\contoc(V,\alpha)$ denotes the set of contactomorphisms isotopic to the identity.
\begin{question}\label{q:zoll}
    Given a closed contact manifold endowed with a contact form $(V,\alpha)$,
    does every $\varphi\in\contoc(V,\alpha)$ satisfy (\ref{eq:Sandon})?
\end{question}
Similarly to the Arnol'd conjecture, one can ask weaker estimates on the number
of translated points, using the cup-length estimate or the category estimate
of the Lusternik-Schnirelmann theory.
Another variation of Question~\ref{q:zoll} can be asked for the
non-degenerate contactomorphisms (see below), replacing the estimate on the minimal
number of critical points of any $f:M\to\R$ by an estimate on the
minimal number of critical points a Morse map can have.
As mentioned earlier, Sandon proved that Question~\ref{q:zoll} is true for
$(\RP^{2n-1},\alpha_0)$ \cite{San13}.
We proved that this is also the case for every quotient $L^{2n-1}_k(w)$ of
$(\sphere{2n-1},\alpha_0)$
by free $\Z/k\Z$-actions of the form $(z_j)\mapsto (e^{2i\pi w_j/k}z_j)$
(with $k\geq 2$) in \cite{minimalLens}, improving an estimate of
Granja-Karshon-Pabiniak-Sandon \cite{GKPS}.
The question was also positively answered by the work of Albers-Fuchs-Merry \cite{AFM15}
completed by Meiwes-Naef \cite{MN18} in the case of hypertight contact manifolds
(\emph{i.e.} such that all Reeb orbits are non-contractible for some contact form
supporting the contact structure) for generic contactomorphisms
with a Morse estimate.

This first question seems better suited for contact manifolds
all of whose Reeb orbits are closed: the so-called Besse contact manifolds.
A result of Granja-Karshon-Pabiniak-Sandon suggests a second question
that could cover a greater class of contact manifolds.
\begin{question}\label{q:general}
    Given a closed contact manifold endowed with a contact form
    $(V,\alpha)$, does every $\varphi\in\contoc(V,\alpha)$ possess a sequence
    of couples $(p_n,t_n)\in V\times \R$ with $t_n\to+\infty$ such that
    $p_n$ is a translated point of $\varphi$ with time-shift $t_n$?
\end{question}
Let us remark that the Weinstein conjecture is satisfied by the $(V,\alpha)$'s
satisfying the conclusion of Question~\ref{q:general}
(by taking $\varphi=\phi_1^\alpha$).
Let us further remark that a Besse contact manifold answering Question~\ref{q:zoll}
positively also answers Question~\ref{q:general} positively.
Granja-Karshon-Pabiniak-Sandon answered positively Question~\ref{q:general} for
$(L^{2n-1}_k(w),\alpha)$, $\alpha$ being any contact form supporting
the same contact structure as the quotient of $\alpha_0$.

In this paper, we want to give motivations that Question~\ref{q:zoll}
should be answered positively for unit tangent bundles of
Riemannian manifolds, all of whose geodesics are closed and of same prime length
whereas Question~\ref{q:general} should be answered positively
for every unit tangent bundle.
In order to do so, we will prove that this is indeed the case 
in a weak sense for
a subclass of contactomorphisms of $SM$: the contactomorphic
lift of the diffeomorphisms of $M$.

Let $M$ be a Riemannian manifold, its unit tangent bundle $SM$
is a contact manifold for the contact form $\alpha$:
\begin{equation*}
    \alpha_{(x,v)}\cdot\xi = \langle v, \ud\pi_x\cdot\xi\rangle,
    \quad \forall (x,v)\in SM,\forall \xi\in T_{(x,v)}SM,
\end{equation*}
where $\pi:SM\to M$ is the bundle map and $\langle\cdot,\cdot\rangle$
is the Riemannian metric.
Given a diffeomorphism $f:M\to M$,
we denote $\tilde{f}:SM\to SM$ the associated contactomorphism
\begin{equation*}
    \tilde{f}(x,v) := \left(f(x), 
    \frac{\ud f_x^{-T} \cdot v}{\|\ud f_x^{-T} \cdot v\|}\right),\quad
    \forall (x,v)\in SM,
\end{equation*}
where $\ud f^{-T}$ denotes the inverse of the adjoint $\ud f^T$.
We will study the minimal number of translated points of $\tilde{f}$
for $f$ homotopic to the identity.
The notion of translated points of $\tilde{f}$ with time-shift $t$
can be naturally generalized to smooth maps $f:M\to M$ 
(and not only diffeomorphisms) in such a way
that, in particular, a translated point of a diffeomorphism $f:M\to M$
with time-shift
$t$ is exactly a translated point of the contactomorphism
$\tilde{f}$ with time-shift $t$.
\begin{definition}
    A point $(x,v)\in SM$ is a translated point of $f:M\to M$
    with time-shift $t\neq 0$ if their exists a geodesic $\gamma:[0,1]\to\R$
    of length $|t|$ such that
    \begin{equation*}
        \begin{cases}
            \dot{\gamma}(0)=v\|\dot{\gamma}(0)\| \text{ and }
            \ud f_x^T\cdot\dot{\gamma}(1)=\dot{\gamma}(0)
            & \text{ when } t>0,\\
            \dot{\gamma}(1)=v\|\dot{\gamma}(1)\| \text{ and }
            \ud f_x^T\cdot\dot{\gamma}(0)=\dot{\gamma}(1)
            & \text{ when } t<0;
        \end{cases}
    \end{equation*}
    it is a translated point with time-shift $t=0$ if $\ud f_x^T\cdot v=v$
    (in particular, $f(x)=x$).
\end{definition}
The correspondence between translated points of $\tilde{f}$
and translated points of $f$ is due to the fact that the Reeb flow of $SM$
is the geodesic flow (see \emph{e.g.} \cite[\S 1.3.3]{Pat99}).
We give the following partial answer to Question~\ref{q:general}.

\begin{thm}\label{thm:general}
    Let $M$ be a closed Riemannian manifold that has a finite cover $\widetilde{M}$,
    the singular homology group $H_*(\Lambda \widetilde{M})$
    of the free loop space of which
    is not finitely generated.
    Every smooth map $f:M\to M$ homotopic to the identity admits a sequence
    of couples $(p_n,t_n)\in SM\times(0,+\infty)$ with $t_n\to +\infty$
    such that $p_n$ is a translated point of $f$ with time-shift $t_n$.
\end{thm}

The assumption on $M$ that $H_*(\Lambda \widetilde{M})$ is not finitely generated
is satisfied by a large class of closed Riemannian manifolds
(in fact, the author does not know if there are counter-examples).
On the one hand, when the number of conjugacy classes of $\pi_1(M)$ is infinite,
the group $H_0(\Lambda M)$ is already not finitely generated.
On the other hand, when $\pi_1(M)$ is finite,
Vigué-Poirrier and Sullivan proved that $H_*(\Lambda \widetilde{M})$ is not finitely
generated for the universal cover $\widetilde{M}$ \cite{VPS76}.
Therefore, the only possible counter-examples are among the closed manifolds
that have an infinite fundamental group with a finite number of conjugacy classes.

Our most satisfying answer to Question~\ref{q:zoll} concerns the non-degenerate
case. Let us first  discuss the notion of non-degenerate translated points.
Let $(G_t)$ be the geodesic flow of $TM$ and, for a diffeomorphism $f:M\to M$,
let $\hat{f}:TM\to TM$ be its symplectic lift $\hat{f}(x,v) = (f(x),\ud f^{-T}\cdot v)$.
In our particular case, a translated point $p\in SM$ of a contactomorphic
lift $\tilde{f}$ of a diffeomorphism $f:M\to M$ is non-degenerate
for its time-shift $t\in\R$ if and only if $\ud(G_{-t}\circ \hat{f})(p)$
does not have $1$ as an eigenvalue.
One can check that this definition coincides with the definition given
by Sandon in the general setting \cite{San13}.
Similarly to the definition of translated points,
one can extend this notion to any smooth map $M\to M$ (for simplicity,
we only give the definition for positive time-shifts).

\begin{definition}\label{def:nondegenerate}
    A translated point $(x,v)$ of $f:M\to M$ is non-degenerate
    for the time-shift $t>0$ if the subspace of the Jacobi fields $J$
    along the associated geodesic $\gamma$ satisfying
    \begin{equation*}
        \begin{cases}
            J(1) = \ud f\cdot J(0),\\
            (\ud^2 f\cdot J(0))^T \cdot \dot{\gamma}(1) =
            \dot{J}(0) - \ud f_x^T\cdot \dot{J}(1),
        \end{cases}
    \end{equation*}
    is reduced to $0$
    (the linear morphism $\ud^2 f\cdot J(0)$ denotes
    $u\mapsto \ud^2 f[J(0),u]$, see Section~\ref{se:nondegeneracy}).
\end{definition}

The equivalence of both definitions in the case of diffeomorphisms
is proven in Proposition~\ref{prop:nondegeneracy}.
We give the following partial answer to Question~\ref{q:zoll}
in the non-degenerate case.
Let us recall that
a Riemannian manifold $M$ is called Zoll (of length $\ell$)
if all its geodesics are closed and of the same prime length (equal to $\ell$).

\begin{thm}\label{thm:zoll}
    Let $M$ be a closed Zoll Riemannian manifold
    and let $R:=\Z$ if $M$ is orientable
    and $R:=\Z/2\Z$ otherwise.
    For every smooth map $f:M\to M$ homotopic to the identity with finitely
    many translated points in $SM$ all of which are non-degenerate,
    the number of translated points is not less than
    $\sum_j \beta_j(SM;R)$, where $\beta_j(SM;R) = \operatorname{rank} H_j(SM;R)$
    denotes the $j$-th Betti number of $SM$.
\end{thm}

In the degenerate case, our result is less satisfying as it requires
a $C^0$-closeness assumption.
Let us recall that the cup-length $CL(X;R)\in\N$ of a space $X$
is the maximal $k$ such that $u_1\smile\cdots\smile u_k\neq 0$
for some non-zero $u_j\in H^*(X;R)$ of positive degree.
By a $H^1$-homotopy $(f_s)$ of maps $M\to M$, we will mean that
$t\mapsto f_t(x)$ is in the Sobolev space $H^1([0,1],M)$
for all $x\in M$.
The set of time-shifts of $f:M\to M$ will denote the set of $t\in\R$
that are time-shifts of a translated point;
when $M$ is Zoll of length $\ell$, this set is $\ell\Z$-invariant
so it can be seen as a subset of $\R/\ell\Z$.

\begin{thm}\label{thm:zolldegenerate}
    Let $M$ be a closed Zoll Riemannian manifold of length $\ell$
    and let $R:=\Z$ if $M$ is orientable
    and $R:=\Z/2\Z$ otherwise.
    Let $(f_s)$ be a $H^1$-homotopy of maps $M\to M$ such that $f_0=\id$ and
    \begin{equation*}
        \int_0^1 \|\partial_t f_t(x)\|^2\ud t < \left(\frac{\ell}{2}\right)^2,\quad
        \forall x\in M.
    \end{equation*}
    If the number of time-shifts of $f_1$ seen in $\R/\ell\Z$ is
    less than $1+CL(SM;R)$, then $f_1$ has infinitely many translated points.
    In particular,
    the number of translated points of $f_1$ is not less than
    $1+CL(SM;R)$.
\end{thm}

The only closed Zoll Riemannian manifolds known by the author are diffeomorphic
to the compact rank-one symmetric spaces:
$\sphere{n}$, $\RP^n$, $\CP^n$, $\HP^n$ and $\CaP^2$
(see however \cite{Ber77} for examples of manifolds with exotic
structures all of whose geodesics starting from a special point
go back to this point at the same length).
According to a result of Bott and Samelson, every Zoll Riemannian
manifold has a cohomology ring isomorphic to the cohomology ring
of one of these spaces \cite{Bot54,Sam63}.
Let us refer to \cite{Bes78} for a comprehensive introduction to
the theory of Zoll and Besse Riemannian manifolds.
Let us also point out that
the study of Besse Riemannian manifolds and more generally of Besse
contact forms 
from the variational viewpoint has recently
known significant advances \cite{AB19,GGM21,MS21}.

The proofs of this article are based on a rather classical variational
principle that goes back to Grove \cite{Gro73}.
Indeed, we remark in Section~\ref{se:defvariational} that
geodesics corresponding to translated points of positive time-shift
are exactly the critical points of the restriction of the energy functional
to paths $\gamma:[0,1]\to M$ such that $f(\gamma(0))=\gamma(1)$.
This variational principle was initially applied in the specific case
where $f$ is an isometry in order to study isometry-invariant geodesics
\cite{Gro73,GT78,Maz14,Maz15,Ban16,MM17}.

Theorem~\ref{thm:general} is a direct application of Morse theory.
Theorems~\ref{thm:zoll} and \ref{thm:zolldegenerate} are more
subtle consequences of Morse and Lusternik-Schnirelmann theories.
In order to translate the symmetry $(p,t)\mapsto (p,t+\ell)$
of the set of couples (translated point, associated time-shift),
we make use of the Chas-Sullivan product \cite{CS04}.
More precisely, we apply the results of Goresky-Hingston
concerning the product-structure of the homology groups
of the free loop space of Zoll Riemannian manifolds \cite{GH09}.

\subsection*{Organization of the paper}
In Section~\ref{se:variational}, we study the variational principle
satisfied by translated points of positive time-shifts
and prove Theorem~\ref{thm:general}.
In Section~\ref{se:zoll}, we study the special case of Zoll Riemannian
manifolds and prove Theorems~\ref{thm:zoll} and \ref{thm:zolldegenerate}.

\subsection*{Acknowledgment}
I am grateful to Margherita Sandon and to my former advisor Marco
Mazzuchelli for their supports and fruitful discussions and suggestions.

\section{The variational principle}
\label{se:variational}

\subsection{The variational principle}
\label{se:defvariational}

Let $M$ be a closed Riemannian manifold.
We will denote $PM := H^1([0,1],M)$ the path space of $M$
endowed with its usual structure of Hilbert manifold
(here $H^1$ denotes the Sobolev space also denoted $W^{1,2}$).
Given $f:M\to M$, let
\begin{equation*}\label{eq:Lf}
    \Lambda(f) := \left\{ \gamma\in PM |\
        \gamma(1) = f(\gamma(0)) \right\};
\end{equation*}
in particular $\Lambda(\id)=\Lambda M$ is the free loop space of $M$.
We will also use the notation $\Lambda_f$ for $\Lambda(f)$ and
$\Lambda$ for $\Lambda M$.
This space was already introduced in \cite{Gro73} as $\Lambda_{G(f)} M$,
we refer to this article for details in the properties recalled here.
The space $\Lambda(f)$ is a Hilbert-submanifold as
$\gamma\mapsto (\gamma(0),\gamma(1))$ is a submersion.
Given $\gamma\in\Lambda(f)$, a vector $U\in T_\gamma\Lambda(f)$
is a $H^1$-vector field along $\gamma$ satisfying
$\ud f\cdot U(0) = U(1)$.
Let us denote $E:PM \to\R$ the energy functional,
\begin{equation*}\label{eq:E}
    E(\gamma) := \int_0^1 \|\dot{\gamma}\|^2 \ud t,
\end{equation*}
and $E_f : \Lambda(f)\to\R$ the restriction to $\Lambda(f)$
for $f:M\to M$.
According to \cite[Theorem~2.4]{Gro73}, the functional $E_f$ satisfies the
Palais-Smale condition.

Let us denote $\nabla$ the Levi-Civita covariant derivative of $M$.
Taken along a curve $t\mapsto c(t)$, the covariant derivative of the vector field $X$
will be either denoted $\nabla_{\dot{c}} X$,
$\frac{DX}{\partial t}$ or $\dot{X}$ (this last notation
is reserved to the time variable $t$).
By convention, a geodesic will always mean a geodesic with constant speed:
$\gamma\in PM$ is a geodesic if and only if
\begin{equation*}
    \nabla_{\dot{\gamma}}\dot{\gamma} = \ddot{\gamma} = 0.
\end{equation*}

\begin{prop}
    Given $f:M\to M$, $\gamma\in\Lambda(f)$ is a critical point of
    $E_f$ if and only if $\gamma$ is a geodesic and
    $\ud f_{\gamma(0)}^\bot\cdot \dot{\gamma}(1) = \dot{\gamma}(0)$.
\end{prop}

\begin{proof}
    Let $\gamma\in\Lambda(f)$, let $U\in T_\gamma\Lambda(f)$ and
    let $(\gamma_u)$ be a smooth family in $\Lambda(f)$ such that
    $U=\partial_u\gamma_u$ (the derivative being taken at $u=0$).
    Then
    \begin{equation*}
        \frac{1}{2}\ud E(\gamma)\cdot U = \int_0^1
        \left\langle \frac{D}{\partial u}\dot{\gamma}_u,\dot{\gamma}\right\rangle\ud t
        = \int_0^1
        \left\langle \dot{U},\dot{\gamma}\right\rangle\ud t
        = \big[\langle U,\dot{\gamma}\rangle\big]_0^1 -
        \int_0^1
        \left\langle U,\ddot{\gamma}\right\rangle\ud t.
    \end{equation*}
    As $\ud f\cdot U(0) = U(1)$, one has
    \begin{equation*}
        \big[\langle U,\dot{\gamma}\rangle\big]_0^1 =
        \langle U(0),\ud f^T\cdot\dot{\gamma}(1) - \dot{\gamma}(0)\rangle.
    \end{equation*}
    The identity $\ud E_f(\gamma)\cdot U = 0$ for the vector fields along $\gamma$
    such that $U(0)=0$ implies that $\ddot{\gamma} = 0$, \emph{i.e.}
    $\gamma$ is a geodesic.
    The identity $\ud E_f(\gamma)\cdot U = 0$ for every $U\in T_\gamma\Lambda(f)$
    then implies that $\ud f^T\cdot\dot{\gamma}(1) = \dot{\gamma}(0)$
    (although we will often write $\ud f^T$ without mention of the base point
    of $M$ at which we take the adjoint, we point out that this notation
    can be ambiguous as $\ud f^T$ is define on $f^*TM$ rather than
    $TM$).
\end{proof}

\begin{cor}\label{cor:variational}
    For every diffeomorphism $f:M\to M$,
    critical points of $E_f$ with energy value $e>0$ are in bijection
    with translated points $(x,v)$ of $\tilde{f}$ with
    time-shift $\sqrt{e}$.
\end{cor}

As the correspondence suggests, we will rather be interested in the
values of the functional $P_f := \sqrt{E_f}$ (which is $C^1$ away from
$\{ E_f = 0\}$) as they correspond to time-shifts at critical points
with positive critical values.
For every $\lambda\geq 0$, let $\Lambda_{f}^{<\lambda}$ and
$\Lambda_{f}^{\leq\lambda}$ be the respective subsets $\{ P_f <\lambda\}$
and $\{ P_f \leq \lambda\}$ of $\Lambda(f)$;
when $\lambda^2$ is not a critical value of $E_f$, these
subsets are submanifolds (with boundary in the second case).

\subsection{The topology of $\Lambda(f)$}

Let us first extend the join of loops defined in \cite[\S 2.3]{GH09}
to the spaces $\Lambda(f)$, for $f:M\to M$ smooth, in the obvious way.
Let us define the submanifold
\begin{equation*}
    PM\times_M PM := \{ (\alpha,\beta)\in PM\times PM\ |\
        \alpha(1) = \beta(0) \}
\end{equation*}
as well as $\Lambda(f)\times_M\Lambda(g):=
(\Lambda(f)\times\Lambda(g))\cap (PM\times_M PM)$.
The concatenation $\phi : PM\times_M PM\to PM$ will denote the following continuous map:
\begin{equation*}
    \phi(\alpha,\beta)(t) :=
    \begin{cases}
        \alpha(\frac{t}{s}), & t\in [0,s] \\
        \beta(\frac{t-s}{1-s}), & t\in [s,1]
    \end{cases},\quad
    \text{where}\quad s = \frac{\sqrt{E(\alpha)}}{\sqrt{E(\alpha)}
        + \sqrt{E(\beta)}},
\end{equation*}
in the case $E(\alpha)=E(\beta)=0$, both paths $\alpha$ and $\beta$
are constant $\equiv p$ and one sets $\phi(\alpha,\beta)\equiv p$.
This map satisfies 
\begin{equation}\label{eq:Ephi}
\sqrt{E(\phi(\alpha,\beta))} = \sqrt{E(\alpha)} + \sqrt{E(\beta)}.
\end{equation}
The concatenation $\phi$ is associative:
for every $\alpha,\beta,\gamma\in PM$ such that
$\alpha(1)=\beta(0)$ and $\beta(1)=\gamma(0)$,
\begin{equation}\label{eq:phiassociative}
    \phi(\phi(\alpha,\beta),\gamma)=\phi(\alpha,\phi(\beta,\gamma)).
\end{equation}
By restriction, one gets a continuous map
$\phi:\Lambda(f)\times_M \Lambda(g) \to \Lambda(g\circ f)$ satisfying
$P_{g\circ f}(\phi(\alpha,\beta)) = P_f(\alpha) + P_g(\beta)$,
for every $f,g:M\to M$.

Let $(f_s)_{s\in[0,1]}$ be a $H^1$-homotopy of smooth maps $M\to M$,
\emph{i.e.} such that
$s\mapsto f_s(x)$ is in $PM$ for all $x\in M$.
It induces the following map $\tau(f_s):\Lambda(f_0)\to\Lambda(f_1)$,
\begin{equation*}
    \tau(f_s)(\alpha) := \phi\big(\alpha,t\mapsto f_t(\alpha(0))\big).
\end{equation*}
Let 
\begin{equation*}
\delta(f_s) := \sup_{x\in M} \sqrt{E(t\mapsto f_t(x))},
\end{equation*}
then $\tau(f_s) : \Lambda_{f_0}^{\leq \lambda} \to 
\Lambda_{f_1}^{\leq \lambda + \delta(f_s)}$ according
to (\ref{eq:Ephi}).
Following the same lines as \cite[Lemma~3.6]{Gro73}
for the more usual concatenation of continuous paths,
one proves the following counterpart for $\phi$.

\begin{lem}[{\cite[Lemma~3.6]{Gro73}}]\label{lem:taufs}
    Let $(f_s)_{s\in[0,1]}$ be a homotopy of smooth maps $M\to M$
    such that $s\mapsto f_s(x)$ is in $PM$ for all $x\in M$.
    Then $\tau(f_s)$ and $\tau(f_{1-s})$ are
    homotopy inverses.
\end{lem}

In particular, when $f:M\to M$ is homotopic to identity,
$\Lambda(f)$ is homotopy equivalent to the free loop space
$\Lambda M$ whereas when $f:M\to M$ is homotopic to a constant,
$\Lambda(f)$ is homotopy equivalent to a point.

\begin{cor}\label{cor:taufs}
    Under the hypothesis of Lemma~\ref{lem:taufs},
    the maps $\tau(f_s)$ and $\tau(f_{1-s})$ induce
    a $\delta(f_s)$-interleaving between the persistence modules
    $t\mapsto H_*(\Lambda_{f_j}^{<t})$, $j\in\{ 0,1\}$, \emph{i.e.},
    for all $t\geq 0$, the induced morphisms
    \begin{equation*}
        \begin{cases}
            \tau(f_{s})_* : H_*(\Lambda_{f_0}^{< t}) \to
            H_*(\Lambda_{f_1}^{<t+\delta(f_s)}), \\
            \tau(f_{1-s})_* : H_*(\Lambda_{f_1}^{< t}) \to
            H_*(\Lambda_{f_0}^{<t+\delta(f_s)}),
        \end{cases}
    \end{equation*}
    commute with the inclusion morphisms
    (the same is true for $t\mapsto H_*(\Lambda_{f_j}^{\leq t})$).
\end{cor}

\subsection{Proof of Theorem~\ref{thm:general}}

The following fact is well-known to the experts when $f=\id$
(and this case will be enough for us).

\begin{lem}\label{lem:inclusionfinite}
    Given $f:M\to M$, for every $\lambda\geq 0$, the image of the
    inclusion morphism $H_*(\Lambda_f^{\leq \lambda})\to H_*(\Lambda_f)$
    is finitely generated.
\end{lem}

\begin{proof}
    Let $f:M\to M$. It is enough to prove that for every $\lambda\geq 0$,
    the inclusion morphism $H_*(\Lambda_f^{\leq \lambda})\to H_*(\Lambda_f^{\leq K})$
    is finitely generated for some $K\geq\lambda$.
    Let $K>0$ be a regular value of $E_f$. Using a subspace of broken geodesics,
    one can retract $\Lambda_f^{\leq K}$ on a finite-dimensional submanifold
    $N$ with a retraction $r:\Lambda_f^{\leq K}\to N$ by deformation such that
    $E_f\circ r\leq E_f$ and the critical points of $E_f$ are exactly
    the critical points of $g:=E_f\circ r$
    (see \emph{e.g.} \cite[\S 16]{Mil63}).
    Therefore, the desired result boils down to proving that the image of
    the homology morphism induced by $\{ g\leq \lambda\}
    \hookrightarrow N$ has a finitely generated image for $g:N\to\R$ a map
    on a finite-dimensional manifold satisfying the Palais-Smale condition.
    As this property is $C^0$-open, one can assume that $g$ is a Morse
    map satisfying the Palais-Smale condition
    (one can adapt indeed the proof of \cite[Theorem~2.7]{Mil65}
    using the fact that the critical set of $g$ in $\{ g\leq\lambda\}$
    is compact and the norm of $\ud g$ is uniformly bounded
    from below outside any neighborhood of this set).
    The Palais-Smale condition then
    implies that $\{ g\leq\lambda\}$ contains only a finite number $m\in\N$
    of critical points (which are non-degenerate), so $H_*(\{ g\leq\lambda\})$
    can be generated by
    $m$ elements.
\end{proof}

\begin{proof}[Proof of Theorem~\ref{thm:general}]
    Let us first prove that one can assume $M=\widetilde{M}$.
    Let us assume that the theorem is true for a finite Riemannian cover
    $q:\widetilde{M}\to M$
    satisfying the homological hypothesis.
    Let $f=f_1:M\to M$ be homotopic to the $\id=f_0$ through $(f_t)$.
    By applying the homotopy lifting property to $(f_t\circ q)$
    one gets a homotopy $(\tilde{f}_t)$ in $\widetilde{M}\to\widetilde{M}$
    from the identity to a map $\tilde{f}_1$ commuting with $f$.
    If $(\tilde{p}_n,t_n)$ is a sequence of $\widetilde{M}\times\R$
    satisfying the conclusion of the theorem for $\tilde{f}_1$,
    then $(q(\tilde{p}_n),t_n)$ is the desired sequence for $f$.

    Let us now assume that $M=\widetilde{M}$.
    Let $f:M\to M$ be a smooth map homotopic to the identity.
    According to Lemma~\ref{lem:taufs},
    $H_*(\Lambda(f))$ is isomorphic to $H_*(\Lambda M)$ which is
    not finitely generated.
    Therefore, Lemma~\ref{lem:inclusionfinite} implies
    that $H_*(\Lambda_f^{\leq\lambda})\to H_*(\Lambda_f)$ is never
    onto for $\lambda\in [0,+\infty)$.
    According to the Morse deformation lemma,
    there exists a sequence $(t_n)$ of positive critical values
    with $t_n\to +\infty$.
    The conclusion follows from Corollary~\ref{cor:variational}.
\end{proof}

\subsection{Non-degeneracy of a translated point}
\label{se:nondegeneracy}

In this section, we show the equivalence between the non-degeneracy
of a translated point in the sense of contact geometry and
the non-degeneracy of the associated critical point of $E_f$.

Let us denote $R$ the Riemann tensor defined by
\begin{equation*}\label{eq:Riemann}
    R(X,Y) := [\nabla_X,\nabla_Y] - \nabla_{[X,Y]},
\end{equation*}
so that a vector field $J$ along $\gamma$ is a Jacobi field
if and only if
\begin{equation*}
    \ddot{J} = R(\dot{\gamma},J)\dot{\gamma}.
\end{equation*}
Let us also recall that the second derivative $\ud^2 f$ of a map
$f:M\to N$ is well-defined as $\nabla(\ud f)$,
\emph{i.e.} by the tensorial expression: for all vector fields $X,Y$,
\begin{equation*}
    \ud^2 f[X,Y] := X\cdot (Y\cdot f) - \ud f\cdot \nabla_X Y.
\end{equation*}

Let us recall that the Hessian $\ud^2 E_f(\gamma)$ of
the energy functional is well-defined at a critical point $\gamma$
(see \emph{e.g.} \cite[\S 13]{Mil63}).

\begin{prop}\label{prop:hessian}
    Let $\gamma\in\Lambda(f)$ be a critical point of $E_f$,
    for every $U,V\in T_\gamma\Lambda(f)$,
    \begin{equation*}
        \frac{1}{2}\ud^2 E_f(\gamma)[U,V] =
        \int_0^1 \left[R(V,\dot{\gamma})U,\dot{\gamma}) +
        \langle\dot{U},\dot{V}\rangle\right]\ud t
        + \langle \ud^2 f_{\gamma(0)}[U(0),V(0)],\dot{\gamma}(1)\rangle.
    \end{equation*}
\end{prop}

\begin{proof}
    Let $U,V\in T_\gamma\Lambda(f)$ and
    let $(\gamma_{u,v})$ be a smooth family of $\Lambda(f)$ such that
    $\gamma_{0,0} = \gamma$, $U = \partial_u\gamma_{u,0}$,
    $V = \partial_v\gamma_{0,v}$ (the partial derivatives being
    taken at $0$).
    Then
    \begin{equation*}
        \begin{split}
            \frac{1}{2}\ud^2 E_f(\gamma)[U,V]
            &=
        \int_0^1 \frac{\partial}{\partial v}
        \left\langle \frac{D}{\partial
        u}\dot{\gamma}_{u,v},\gamma_{0,v}\right\rangle \ud t \\
            &=
            \int_0^1 \left[\left\langle \frac{D^2}{\partial v\partial u}\dot{\gamma}_{u,v},
            \dot{\gamma}\right\rangle + \left\langle \frac{D}{\partial u}\dot{\gamma}_{u,0},
            \frac{D}{\partial v}\dot{\gamma}_{0,v}\right\rangle \right]\ud t \\
            &=
            \int_0^1 \left[\left\langle
                R(V,\dot{\gamma})U,\dot{\gamma}\right\rangle + \langle
    \dot{U},\dot{V}\rangle \right]\ud t +
    \big[\langle \nabla_V U,\dot{\gamma} \rangle\big]^1_0.
    \end{split}
    \end{equation*}
    By derivating the identity $f(\gamma_{u,v}(0)) = \gamma_{u,v}(1)$,
    one gets
    \begin{equation*}
        \ud^2 f_{\gamma(0)}[V(0),U(0)] + \ud f_{\gamma(0)}\cdot\nabla_V U(0) =
        \nabla_V U(1),
    \end{equation*}
    so that, using $\ud f^T\cdot\dot{\gamma}(1) = \dot{\gamma}(0)$,
    \begin{equation*}
        \big[\langle \nabla_V U,\dot{\gamma} \rangle\big]^1_0 =
        \langle \nabla_V U(1) - \ud f\cdot\nabla_V U(0),\dot{\gamma}(1)\rangle
        = \langle \ud^2 f[U(0),V(0)] , \dot{\gamma}(1) \rangle.
    \end{equation*}
\end{proof}

\begin{cor}\label{cor:hessian}
    Let $\gamma\in\Lambda(f)$ be a critical point of $E_f$,
    a vector field $J\in T_\gamma\Lambda(f)$ belongs to the kernel
    of the quadratic form $\ud^2 E_f(\gamma)$ if and only if
    $J$ is a Jacobi field satisfying
    \begin{equation*}
        (\ud^2 f\cdot J(0))^T \cdot\dot{\gamma}(1) =
        \dot{J}(0) - \ud f^T\cdot \dot{J}(1),
    \end{equation*}
    where $\ud^2 f\cdot J(0)$ denotes the linear morphism
    $T_{\gamma(0)}M\to T_{\gamma(1)} M$, $u\mapsto \ud^2 f[J(0),u]$.
\end{cor}

In particular, the kernel of $\ud^2 E_f(\gamma)$ has a dimension bounded
by $2\dim M - 1$.
Indeed, the vector space of Jacobi fields of $\gamma$ has dimension $2\dim M$,
and $J\in T_\gamma\Lambda(f)$ implies that $J(1)=0$ if $J(0)=0$,
so, for instance, the Jacobi field $J(t)=t\dot{\gamma}(t)$ does not belong
to $T_\gamma\Lambda(f)$.

\begin{proof}
    Let $J\in \ker \ud^2 E_f(\gamma)$.
    Applying Proposition~\ref{prop:hessian} with $U:=J$ and every
    vector field $V$ with $V(0)=0$, one proves that $J$ is a Jacobi
    field in a classical manner.
    The conclusion now follows from the identity $\forall V\in T_\gamma\Lambda(f)$,
    \begin{equation*}
        \int_0^1 \langle \dot{J},\dot{V}\rangle\ud t =
        [\langle \dot{J},V\rangle]_0^1 - \int_0^1 \langle \ddot{J},V\rangle\ud t
        = \langle \ud f^T\cdot \dot{J}(1) - \dot{J}(0), V(0)\rangle
        - \int_0^1 \langle \ddot{J},V\rangle\ud t.
    \end{equation*}
\end{proof}

\begin{prop}\label{prop:nondegeneracy}
    Let $f:M\to M$ be a diffeomorphism, a translated point $(x,v)\in SM$
    of $\tilde{f}$ is non-degenerate for the time-shift $t>0$
    if and only if the associated geodesic $\gamma\in\Lambda(f)$
    of length $t$ is a non-degenerate critical point of $E_f$.
\end{prop}

\begin{proof}
    We will need the following differential identity: for
    every diffeomorphism $f:M\to M$ and every
    vector field $U$ of $M$,
    \begin{equation}\label{eq:diffT}
        \nabla_U(\ud f^{-T}) = - \ud f^{-T}\cdot(\nabla_U\ud f)^T\cdot\ud f^{-T}.
    \end{equation}
    This identity can be derived by taking the covariant derivative of the identity
    $\ud f^{-T}\cdot\ud f^T = \id$ ($\id$ meaning the section
    $x\mapsto \id_{T_x M}$)
    and using $\nabla_U(\ud f^T) = (\nabla_U \ud f)^T$
    (which can be obtained by derivating the definition of the adjoint operator).

    Let $f:M\to M$ be a diffeomorphism and let
    $\hat{f}:TM\to TM$ be its symplectic lift,
    \begin{equation*}
        f(x,v) = (f(x),\ud f_x^{-T}\cdot v),\quad
        \forall (x,v)\in TM.
    \end{equation*}
    The differential of $\hat{f}$ at $(x,v)\in TM$ is
    \begin{equation*}
        \ud\hat{f}_{(x,v)}\cdot (\xi_1,\xi_2) =
        \left(\ud f_x\cdot\xi_1,\ud f_x^{-T}\cdot\xi_2 + \nabla_{\xi_1}(\ud
        f^{-T})\cdot v\right),\quad \forall \xi_1,\xi_2\in T_x M,
    \end{equation*}
    where the identification $T_{(x,v)}TM\simeq T_x M\times T_x M$
    is given by the Levi-Civita connection (see \emph{e.g.} \cite[\S 1.3.1]{Pat99}).
    Let $(x,v)\in SM$ be a translated point of $\tilde{f}$ for the
    time-shift $t>0$ and $\gamma\in\Lambda(f)$ be the associated
    geodesic of length $t$.
    It is non-degenerate if and only if 
    $\ud(\hat{f}\circ G_{-t})_{G_t(x,v)}$
    does not have $1$ as an eigenvalue.
    We recall that, for $t>0$,
    \begin{equation*}
        (\ud G_{-t})_{G_t(x,v)}\cdot\left(J(1),\frac{1}{t}\dot{J}(1)\right) =
        \left(J(0),\frac{1}{t}\dot{J}(0)\right),
    \end{equation*}
    for every Jacobi field $J$ along $\gamma$
    (the $t^{-1}$ factor being due to the reparametrization, see \emph{e.g.}
    \cite[\S 1.5]{Pat99}).
    Since $J\mapsto (J(1),\dot{J}(1))$ is an isomorphism and
    $tv = \dot{\gamma}(0)$, $(x,v)$ is degenerate
    if and only if there exists a Jacobi field $J$ along $\gamma$ such that
    \begin{equation}\label{eq:degenerateCont}
        \begin{cases}
            \ud f\cdot J(0) = J(1), \\
            \ud f^{-T}\cdot \dot{J}(0) + \nabla_{J(0)}(\ud f^{-T})\cdot \dot{\gamma}(0)
            = \dot{J}(1).
        \end{cases}
    \end{equation}
    The first equation of (\ref{eq:degenerateCont}) means that $J\in T_\gamma\Lambda(f)$.
    Using identity (\ref{eq:diffT}) and
    $\ud f^{-T}\cdot\dot{\gamma}(0) = \dot{\gamma}(1)$, the second equation
    of (\ref{eq:degenerateCont}) becomes
    \begin{equation*}
        \ud f^{-T}\cdot\dot{J}(0) - \ud f^{-T}\cdot (\nabla_{J(0)}\ud f)^T\cdot
        \dot{\gamma}(1) = \dot{J}(1).
    \end{equation*}
    Finally,
    applying $\ud f^T$ to this last equation, we find the equation of
    Corollary~\ref{cor:hessian}.
\end{proof}

\section{The Zoll case}
\label{se:zoll}

\subsection{The Chas-Sullivan product}

Let us recall and extend in an obvious manner
the filtered Chas-Sullivan product defined in \cite{GH09}.
We refer to \cite{GH09} for technical details.
For now, we do not need to assume that the closed Riemannian manifold $M$ is Zoll.
Let $f,g:M\to M$ be smooth maps, as the space $\Lambda(f)\times_M\Lambda(g)$
is a submanifold of $\Lambda(f)\times\Lambda(g)$ of codimension $n:=\dim M$,
there is a well-defined Gysin morphism
\begin{equation*}
    H_*(\Lambda(f)\times\Lambda(g))\to H_{*-n}(\Lambda(f)\times_M\Lambda(g))
\end{equation*}
(see \emph{e.g.} \cite[Proposition~B.2]{GH09}).
The Chas-Sullivan product is defined by composing this morphism
with the concatenation $\phi$:
\begin{equation*}
    * : H_*(\Lambda(f))\otimes H_*(\Lambda(g)) \to
    H_{*-n}(\Lambda(g\circ f)).
\end{equation*}
Let $a,b>0$ be regular values of $P_f$ and $P_g$,
then $\Lambda_{f}^{\leq a}$, $\Lambda_{g}^{\leq b}$
are submanifolds with boundary and one can thus define the Gysin morphism
\begin{equation*}
    H_*(\Lambda_{f}^{\leq a}\times\Lambda_{g}^{\leq b})\to 
    H_{*-n}(\Lambda_{f}^{\leq a}\times_M\Lambda_{g}^{\leq b}).
\end{equation*}
According to (\ref{eq:Ephi}), one can thus define a filtered Chas-Sullivan product
\begin{equation*}
    * : H_*(\Lambda_{f}^{\leq a})\otimes H_*(\Lambda_{g}^{\leq b}) \to
    H_{*-n}(\Lambda_{g\circ f}^{\leq a+b})
\end{equation*}
commuting with the inclusion morphisms.
A relative version of this product is also available
\begin{gather*}
    H_*(\Lambda_{f}^{\leq a},\Lambda_{f}^{\leq a'})\otimes 
    H_*(\Lambda_{g}^{\leq b},\Lambda_{g}^{\leq b'}) \to
    H_{*-n}(\Lambda_{g\circ f}^{\leq a+b},
    \Lambda_{g\circ f}^{\leq \max(a+b',a'+b)}), \\
    H_*(\Lambda_{f}^{\leq a},\Lambda_{f}^{< a})\otimes 
    H_*(\Lambda_{g}^{\leq b},\Lambda_{g}^{< b}) \to
    H_{*-n}(\Lambda_{g\circ f}^{\leq a+b},
    \Lambda_{g\circ f}^{< a+b}),
\end{gather*}
that commutes with each other through inclusion morphisms.

\begin{lem}\label{lem:taustar}
    Let $g:M\to M$ be a smooth map,
    let $(f_s)$ be an $H^1$-homotopy of smooth maps $M\to M$,
    Let $\delta := \delta(f_s)$ (which is $\geq \delta(f_s\circ g)$).
    For every $\alpha\in H_*(\Lambda_g^{\leq a})$,
    $\beta\in H_*(\Lambda_{f_0}^{\leq b})$,
    \begin{equation*}
    \alpha * \tau(f_s)_*\beta = \tau(f_s\circ g)_*(\alpha *\beta) \in H_*(\Lambda_{f_1\circ
        g}^{\leq a+b+\delta}),
    \end{equation*}
    as long as the product $*$ is well-defined.
\end{lem}

\begin{proof}
    Let us consider the following diagram:
    \begin{equation*}
        \begin{gathered}
            \xymatrix{
                \Lambda^{\leq a}_g\times \Lambda_{f_0}^{\leq b}
                \ar[d]^-{\id\times\tau(f_s)}
                & \Lambda_g^{\leq a}\times_M\Lambda_{f_0}^{\leq b}
                \ar@{_{(}->}[l] \ar[d]^-{\id\times\tau(f_s)} \ar[r]^-{\phi}
                & \Lambda_{f_0\circ g}^{\leq a+b} \ar[d]^-{\tau(f_s\circ g)}
                \\
                \Lambda_g^{\leq a}\times\Lambda_{f_1}^{\leq b+\delta}
                & 
                \Lambda_g^{\leq a}\times_M \Lambda_{f_1}^{\leq b+\delta}
                \ar@{_{(}->}[l] \ar[r]^-{\phi}
                &
                \Lambda_{f_1\circ g}^{\leq a+b+\delta}
            }
        \end{gathered},
    \end{equation*}
    where the unlabeled arrows are inclusion maps.
    By definition of $\tau(f_s)$ and $\tau(f_s\circ g)$
    and by associativity of $\phi$ (\ref{eq:phiassociative}),
    this diagram commutes.
    By naturality of the Gysin morphisms and by definition of the product $*$,
    the conclusion follows.
\end{proof}

In the same way, one proves that $*$ is associative.

The product structure of the homology group $H_*(\Lambda)$ was
studied by Goresky-Hingston, especially in the case where all geodesics
are closed and of same prime length $\ell$ \cite[\S 13-15]{GH09}.
In the sequel, $M$ will satisfy this assumption and
the coefficient ring of the singular homology groups
will be $\Z$ if $M$ is orientable and $\Z/2\Z$ otherwise.
In this case, the energy functional $E_\id$ is a perfect Morse-Bott functional:
it implies in particular that the following sequences of inclusion morphisms are exact
\begin{equation}\label{eq:Eperfect}
    \begin{split}
    0\to H_*(\Lambda^{\leq a})\to H_*(\Lambda^{\leq b}) \to
    H_*(\Lambda^{\leq b},\Lambda^{\leq a})\to 0, \\
    0\to H^*(\Lambda^{\leq b},\Lambda^{\leq a})\to H^*(\Lambda^{\leq b})
    \to H^*(\Lambda^{\leq a})\to 0,
    \end{split}
\end{equation}
for every $0\leq a\leq b\leq c\leq +\infty$, with $\Lambda^{\leq+\infty}:=\Lambda$,
the same being true replacing $\Lambda^{\leq\lambda}$ with $\Lambda^{<\lambda}$.
Moreover, the filtered product $*$ can be defined for all values.
The critical values of $P_f$ correspond to the positive multiples of $\ell$
with associated critical submanifolds diffeomorphic to $SM$
\emph{via} $\gamma\mapsto \dot{\gamma}(0)/\|\dot{\gamma}(0)\|$.
Let $\lambda_1\geq 0$ be the Morse index of the critical submanifold $\Sigma_1$ associated
with the critical value $\ell$.
The manifold $\Sigma_1$ can be oriented and one defines
\begin{equation*}
    \Theta\in H_{2n-1+\lambda_1}(\Lambda^{\leq\ell})
\end{equation*}
the image of the fundamental class of $\Sigma_1$ under the canonical isomorphism
\begin{equation}\label{eq:isoSigma}
    H_*(\Lambda^{\leq\ell})\simeq H_*(\Lambda^{\leq 0})\oplus H_{*-\lambda_1}(\Sigma_1)
\end{equation}
resulting from the splitting of the exact sequence (\ref{eq:Eperfect}) for the pair
$(\Lambda^{\leq\ell},\Lambda^{\leq 0})$ induced by $\mathrm{ev}:\Lambda\to\Lambda^{\leq 0}$,
$\mathrm{ev}(\gamma)\equiv \gamma(0)$,
and the Morse-Bott isomorphism
$H_*(\Lambda^{\leq\ell},\Lambda^{\leq 0})\simeq H_{*-\lambda_1}(\Sigma_1)$.
Let us set $b:=\lambda_1+n-1$.

\begin{thm}[{\cite[Theorem~13.4]{GH09}}]\label{thm:GH}
    Let $M$ be an $n$-dimensional closed Zoll Riemannian manifold
    of length $\ell$.
    The product 
    \begin{equation*}
        \Theta*\cdot : H_*(\Lambda,\Lambda^{\leq 0})\to
        H_{*+b}(\Lambda,\Lambda^{\leq 0})
    \end{equation*}
    with the class $\Theta$ is injective
    and induces an isomorphism
    \begin{equation*}
        H_*(\Lambda^{\leq r\ell},\Lambda^{<r\ell})\to H_{*+b}(\Lambda^{\leq (r+1)\ell},
        \Lambda^{<(r+1)\ell})
    \end{equation*}
    for every positive integer $r$.
\end{thm}

\subsection{Min-max critical values}

Given $f:M\to M$,
let us define min-max critical values of $P_f$ associated with
homology classes of $\Lambda(f)$.
Given $\alpha\in H_*(\Lambda(f))$, let us define
\begin{equation*}
c(\alpha,f) := \inf\left\{ \lambda\geq 0\ |\ 
\alpha\in\im\left(H_*(\Lambda^{\leq \lambda}_f)\to H_*(\Lambda_f)\right)\right\}.
\end{equation*}
By the Morse deformation lemma, if $c(\alpha,f)>0$ then it is a critical value
of $P_f$.
A consequence of Theorem~\ref{thm:GH} and the perfectness of $E_\id$
in the Zoll case
is the following corollary.

\begin{cor}\label{cor:cThetakid}
    Let $M$ be a Zoll Riemannian manifold with prime length $\ell$.
    For every $\alpha\in H_*(\Lambda)$, one has $c(\alpha,\id)\in\N\ell$ and
    if $c(\alpha,\id)>0$,
    then
    \begin{equation*}
        c(\Theta^k *\alpha,\id) = c(\alpha,\id) + k\ell,
        \quad \forall k\in\N.
    \end{equation*}
\end{cor}

\begin{proof}
    Since the $c(\alpha,\id)$ are either $0$ or critical values of $P_{\id}$,
    one has $c(\alpha,\id)\in\N\ell$ for all $\alpha\in H_*(\Lambda)$.
    Let $\alpha\in H_*(\Lambda)$ be such that $c(\alpha,\id)>0$,
    so $c(\alpha,\id) = r\ell$ with $r\in\N^*$ and there exists
    $\beta\in H_*(\Lambda^{\leq r\ell})$ the image of which is $\alpha$
    under the inclusion morphism.
    Since $\Theta^k*\cdot$ commutes with inclusion morphisms,
    the image of $\Theta^k*\beta\in H_{*+kb}(\Lambda^{\leq (r+k)\ell})$ in
    $H_*(\Lambda)$ is $\Theta^k*\alpha$,
    so $c(\Theta^k*\alpha,\id)\leq (r+k)\ell$.

    The image $\gamma$ of $\beta$ under the inclusion morphism
    $H_*(\Lambda^{\leq r\ell})\to H_*(\Lambda^{\leq r\ell},\Lambda^{<r\ell})$
    is non-zero since $c(\alpha,\id)\geq r\ell$
    (by looking at the long exact sequence of the couple
    $(\Lambda^{\leq r\ell},\Lambda^{<r\ell})$).
    Let us consider the commuting diagram:
    \begin{equation*}
        \begin{gathered}
            \xymatrix{
                H_*(\Lambda^{\leq r\ell}) \ar[d] \ar[r]^-{\Theta^k*\cdot}
                & H_{*+kb}(\Lambda^{\leq (r+k)\ell}) \ar[d]
                \\
                H_*(\Lambda^{\leq r\ell},\Lambda^{<r\ell})
                \ar[r]^-{\Theta^k*\cdot}_-{\simeq}
                & H_{*+kb}(\Lambda^{\leq (r+k)\ell},\Lambda^{<(r+k)\ell})
            }
        \end{gathered},
    \end{equation*}
    where the vertical arrows are inclusion morphisms.
    Since the bottom arrow is an isomorphism according to Theorem~\ref{thm:GH},
    the image of $\gamma$ is non-zero and so is the image of $\Theta^k*\beta$
    under the right vertical arrow.
    This implies that $\Theta^k*\beta$ is not in the image of
    $H_{*+kb}(\Lambda^{<(r+k)\ell})$ under the inclusion morphism
    so $c(\Theta^k*\alpha,\id)\geq (r+k)\ell$.
\end{proof}

\begin{prop}\label{prop:cThetaf}
    Let $M$ be a Zoll Riemannian manifold with prime length $\ell$
    and let $(f_s)$ be a $H^1$-homotopy from $\id$ to $f$ and
    let us set $\tau := \tau(f_s)$ and $\delta:=\delta(f_s)$.
    \begin{enumerate}
        \item For every $\alpha\in H_*(\Lambda)$,
            $|c(\alpha,\id)-c(\tau_*\alpha,f)| \leq \delta$.
        \item For every $\alpha\in H_*(\Lambda)$,
            \begin{equation*}
                c(\Theta^{k+1} *\tau_*\alpha,f) = c(\Theta^k *\tau_*\alpha,f)
                + \ell - \varepsilon_k(\alpha),\quad
                \forall k\in\N,
            \end{equation*}
            where $(\varepsilon_k(\alpha))$ is a sequence of non-negative
            reals that converges to $0$ when $c(\alpha,\id)>0$.
    \end{enumerate}
\end{prop}

\begin{proof}
    The first statement is a direct consequence of the fact that $\tau(f_s)$
    and $\tau(f_{1-s})$ induce a $\delta$-interleaving in the sense of
    Corollary~\ref{cor:taufs}.
    Let us consider the commuting diagram
    \begin{equation*}
        \begin{gathered}
            \xymatrix{
                H_*(\Lambda_f^{\leq\lambda}) \ar[r]^-{\Theta*\cdot} \ar[d]
                & H_{*+b}(\Lambda_f^{\leq\lambda+\ell}) \ar[d] \\
                H_*(\Lambda_f) \ar[r]^-{\Theta*\cdot}
                & H_{*+b}(\Lambda_f)
            }
        \end{gathered},
    \end{equation*}
    where the vertical arrows are inclusion morphisms and $\lambda\geq 0$.
    It implies that if $c(\alpha,f) < \lambda$ then
    $c(\Theta*\alpha,f) < \lambda + \ell$.
    In particular, for every $\alpha\in H_*(\Lambda)$,
    \begin{equation*}
        c(\Theta^{k+1} *\tau_*\alpha,f) = c(\Theta^k *\tau_*\alpha,f)
        + \ell - \varepsilon_k(\alpha),\quad
        \forall k\in\N,
    \end{equation*}
    for some non-negative sequence $(\varepsilon_k(\alpha))$.
    According to the first statement of this proposition,
    for every $k\in\N$,
    \begin{equation*}
        \sum_{i=0}^{k-1} \varepsilon_i(\alpha) =
        c(\tau_*\alpha,f)-c(\Theta^k*\tau_*\alpha,f) + k\ell
        \leq 2\delta + c(\alpha,\id)-c(\Theta^k*\alpha,\id) + k\ell,
    \end{equation*}
    where we have used that $\Theta^k*\tau_*\alpha = \tau_*(\Theta^k*\alpha)$
    (Lemma~\ref{lem:taustar}).
    Therefore, Corollary~\ref{cor:cThetakid} implies that the series
    $\sum_i \varepsilon_i(\alpha)$ is bounded by $2\delta$
    when $c(\alpha,\id)>0$, bringing the conclusion.
\end{proof}

\subsection{Proof of Theorems~\ref{thm:zoll} and \ref{thm:zolldegenerate}}

\begin{proof}[Proof of Theorem~\ref{thm:zoll}]
    Let $N:=\sum_i \beta_i(SM)$ and let $\alpha_1,\ldots,\alpha_N\in H_*(\Lambda)$
    be an independent family satisfying $c(\alpha_i,\id)=\ell$ for all $i$.
    Such a family can be obtained as follows: one takes the image in
    $H_*(\Lambda^{\leq\ell})$
    of an independent family of $0\oplus H_*(\Sigma_1) \simeq H_*(SM)$
    under the isomorphism (\ref{eq:isoSigma}),
    then takes the image of this family under the inclusion morphism
    $H_*(\Lambda^{\leq\ell})\to H_*(\Lambda)$
    (this last morphism is injective by exactness of (\ref{eq:Eperfect})).
    The exactness of (\ref{eq:Eperfect}) together with Theorem~\ref{thm:GH}
    implies that the family $(\Theta^k * \alpha_i)_{k,i}$, $k\in\N$,
    $1\leq i\leq N$, is independent in $H_*(\Lambda)$.

    Let $(f_s)$ be a $H^1$-homotopy from $\id$ to $f:M\to M$ and
    let $\tau:=\tau(f_s)$.
    Applying the isomorphism $\tau_*:H_*(\Lambda)\to H_*(\Lambda(f))$,
    the family $(\Theta^k*\tau_*\alpha_i)_{k,i}$ is independent in
    $H_*(\Lambda(f))$ (we have used Lemma~\ref{lem:taustar}).
    Let us assume that $f$ has finitely many translated points in $SM$.
    Since $M$ is a Zoll Riemannian manifold,
    the set of positive time-shifts of a translated point of $f$ is
    $t+\N\ell$ for some $t\in (0,\ell]$.
    Therefore, the set of (positive) critical values of $P_f$ is
    invariant by $t\mapsto t+\ell$, discrete such that any interval
    $(a,a+\ell]$, $a\geq 0$, contains a fixed finite number of them.
    Let us study the sequences $(c_k^i) := (c(\Theta^k*\tau_*\alpha_i,f))$
    for $1\leq i\leq N$, the image of which is contained in
    this discrete set.
    Since $c_{k+1}^i = c_k^i + \ell -\varepsilon_k^i$ with
    $\varepsilon_k^i\to 0$ (Proposition~\ref{prop:cThetaf}),
    by discreetness of the $\N\ell$-invariant
    set of critical value, $\varepsilon_k^i = 0$ for large $k$'s.
    We deduce that for $A>0$ large, for each $i$, there exists
    a unique $k_i$ such that $c_{k_i}^i\in (A,A+\ell]$.
    By non-degeneracy of the critical points of $P_f$,
    the Morse inequalities in the window of values $(A,A+\ell]$
    imply that the number of critical points in $P_f^{-1}(A,A+\ell]$
    is not less than the cardinal of $(\Theta^{k_i}*\tau_*\alpha_i)_i$,
    which is $N$.
    For each translated point, only one time-shift belongs to $(A,A+\ell]$,
    which brings the conclusion.
\end{proof}

\begin{proof}[Proof of Theorem~\ref{thm:zolldegenerate}]
    Let $N:= CL(SM)$ and let $u_1,\ldots,u_N \in H^*(\Lambda)$
    be non-zero classes of positive degree such that
    \begin{equation*}
        c(\Theta\frown(u_1\smile\cdots\smile u_N),\id)=\ell.
    \end{equation*}
    Such a family can be obtained as follows:
    let $v_1,\ldots, v_N\in H^*(SM)$ be non-zero classes of positive degree
    such that $v_1\smile\cdots\smile v_N\neq 0$ (which exist by definition of $CL(SM)$),
    then $[SM]\frown(v_1\smile\cdots\smile v_N)\neq 0$
    and we apply (\ref{eq:isoSigma}) together with the exactness of
    (\ref{eq:Eperfect}) to send $[SM]$ to $\Theta$ (by definition of $\Theta$)
    and the $v_i$'s to the $u_i$'s.

    Let $\alpha_0,\ldots,\alpha_N\in H_*(\Lambda)$ be the subordinated classes
    $\alpha_i := \Theta\frown(u_1\smile\cdots\smile u_i)$
    such that $c(\alpha_i,\id)=\ell$ for all $i$
    (as by subordination, $c(\alpha_{i+1},\id)\leq c(\alpha_i,\id)$
    and $c(\Theta,\id)=\ell$).
    Let $(f_s)$ be a $H^1$-homotopy as in the statement of Theorem~\ref{thm:zolldegenerate},
    so that $\delta(f_s) < \ell/2$, and let $\tau:=\tau(f_s)$.
    By subordination, $c(\tau_*\alpha_i,f)$ is non-increasing with $i$
    and by the first point of Proposition~\ref{prop:cThetaf},
    \begin{equation*}
        \frac{\ell}{2} < c(\tau_*\alpha_N,f) \leq c(\tau_*\alpha_{N-1},f)
        \leq \cdots \leq c(\tau_*\alpha_0,f) < \frac{3\ell}{2}.
    \end{equation*}
    According to the Lusternik-Schnirelmann theorem
    (see \emph{e.g.} \cite[\S II.3.2]{Chang}),
    if $P_f$ has a finite number of critical points in the window
    $(\ell/2,3\ell/2)$ then $c(\tau_*\alpha_i,f)$ is decreasing.
    Since the $c(\tau_*\alpha_i,f)$'s are critical values of $P_f$,
    the conclusion follows by Corollary~\ref{cor:variational}.
\end{proof}

\bibliographystyle{amsplain}
\bibliography{biblio} 
\end{document}